\documentclass[11pt]{amsart}
\usepackage{amsmath,amssymb,amsthm,enumerate,xspace,hyperref}

\usepackage{xy} 
\usepackage{mathrsfs} 
\usepackage{graphicx}
\usepackage{color}
\usepackage{csquotes}

\title[Commensurated subgroups]{Commensurated subgroups in tree almost automorphism groups}
\author{Adrien Le Boudec}
\thanks{The first named author is a F.R.S.-FNRS Postdoctoral Researcher and was partially supported by the ERC grant \#278469}
\author{Phillip Wesolek}
\thanks{The second named author is supported by the ERC grant \#278469}

\address{Universit\'{e} catholique de Louvain,
	Institut de Recherche en Math\'{e}matiques et Physique (IRMP),
	Chemin du Cyclotron 2, box L7.01.02,
	1348 Louvain-la-Neuve, Belgium}
\email{adrien.leboudec@uclouvain.be}

\address{ Department of Mathematical Sciences,
	Binghamton University,
	PO Box 6000,
	Binghamton, New York 13902}
\email{pwesolek@binghamton.edu}

\newtheorem{thm}{Theorem}[section]

\newtheorem{prop}[thm]{Proposition}
\newtheorem{lem}[thm]{Lemma}
\newtheorem{cor}[thm]{Corollary}

\theoremstyle{definition}
\newtheorem{defn}[thm]{Definition}
\newtheorem{ex}[thm]{Example}
\newtheorem{quest}[thm]{Question}

\newtheorem{rmk}[thm]{Remark}

\newtheorem*{claim*}{Claim}
\newtheorem*{ack}{Acknowledgments}

\newcommand{\Zb}{\mathbb{Z}}
\newcommand{\Nb}{\mathbb{N}}

\newcommand{\tdlc}{t.d.l.c.\@\xspace}

\newcommand{\mc}[1]{\mathcal{#1}}

\newcommand{\ms}[1]{\mathscr{#1}} 

\newcommand*{\treedk}{\mathcal{T}_{d,k}}
\newcommand*{\bord}{\partial \mathcal{T}_{d,k}}
\newcommand*{\aaut}{\mathrm{AAut}(\mathcal{T}_{d,k})}
\newcommand*{\aautbis}{\mathrm{AAut}(\mathcal{T}_{d,2})}

\newcommand*{\Daaut}{\mathrm{AAut}_{D}(\mathcal{T}_{d,k})}

\newcommand*{\aut}{\mathrm{Aut}(\mathcal{T}_{d,k})}
\newcommand*{\Sd}{\mathrm{Sym}(d)}

\newcommand{\diam}{\mathrm{diam}}

\newcommand{\sleq}{\leqslant}

\newcommand{\Aut}{\mathrm{Aut}}
\newcommand{\Sym}{\mathrm{Sym}}

\newcommand{\normal}{\trianglelefteq}

\newcommand{\Stab}{\mathrm{Stab}}
\newcommand{\rest}{\upharpoonright}
\newcommand{\Sc}{\mc{S}}

\newcommand{\cgrp}[1]{\ol{\langle #1 \rangle}}
\newcommand{\grp}[1]{\langle #1 \rangle}
\newcommand{\ol}[1]{\overline{#1}}

\newcommand{\Trd}{\mathcal{T}_{d,k}}

\begin{document}

\begin{abstract}
We prove that the almost automorphism groups $\aaut$ admit exactly three commensurability classes of closed commensurated subgroups. Our proof utilizes an independently interesting characterization of subgroups of $\aaut$ which contain only periodic elements in terms of the dynamics of the action on the boundary of the tree. 

Our results further cover several interesting finitely generated subgroups of the almost automorphism groups, including the Thompson groups $F,T$, and $V$. We show in particular that Thompson's group $T$ has no commensurated subgroups other than the finite subgroups and the entire group.  As a consequence, we derive several rigidity results for the possible embeddings of these groups into locally compact groups.
\end{abstract}

\maketitle

\setcounter{tocdepth}{1}

\tableofcontents

\addtocontents{toc}{\protect\setcounter{tocdepth}{1}}

\section{Introduction}

The group of almost automorphisms of a locally finite regular tree was first considred by Y. Neretin in \cite{Ner84}. It is a totally disconnected locally compact (\tdlc hereafter) second countable group, which is abstractly simple by work of C. Kapoudjian \cite{Ka99}. 

We briefly recall the definition of this group and its topology, following the approach developed by P.-E. Caprace and T. de Medts in \cite{CdM11}. For $d,k\geq 2$, let $\Trd$ denote the rooted tree in which the root $r$ has valency $k$ and all other vertices have valency $d+1$. The boundary of $\Trd$, denoted $\bord$, is endowed with the visual metric, written $\mathrm{dist}$, corresponding to the root $r$. An \textbf{almost automorphism} of $\Trd$ is a homeomorphism $g \in \mathrm{Homeo}(\bord)$ such that there exists a partition of $\bord$ into finitely many balls $\bord = B_1 \sqcup \ldots \sqcup B_n$ for which $g_{\rest B_i}: B_i \rightarrow g(B_i)$ is a homothety for every $i$. That is to say, for each $i$, there is a constant $\lambda_i$ such that $\mathrm{dist}(g(x),g(y))= \lambda_i \mathrm{dist}(x,y)$ for all $x,y\in B_i$.

The group of automorphisms $\Aut(\Trd)$ of the tree $\Trd$ is a profinite group, and the topology on $\aaut$ is such that $\Aut(\Trd)$ is a compact open subgroup of $\aaut$.

\subsection{Translations, elliptic elements, and locally elliptic subgroups}

An automorphism $g$ of an \textit{unrooted} tree is called elliptic if $g$ stabilizes a vertex or an edge, and it is called hyperbolic if there exists a bi-infinite geodesic line along which $g$ acts by translation. It is a classical result that any automorphism of a tree is either elliptic or hyperbolic. 

The first goal of this paper is to extend the elliptic/hyperbolic typology of tree automorphisms to the almost automorphism setting.

\begin{defn} 
An element $g \in \aaut$ is \textbf{elliptic} if there exists a set-wise $g$-invariant partition $\bord = B_1 \sqcup \ldots \sqcup B_n$ such that $g_{\rest B_i}: B_i \rightarrow g(B_i)$ is a homothety for every $i$.
\end{defn}

\begin{defn}
	An element $g \in \aaut$ is a \textbf{translation} if there exists a ball $B$ of $\bord$ and some $n \in \mathbb{Z}$ such that $g^n_{\rest B}:B\rightarrow g^n(B)$ is a homothety and $g^n(B)\subsetneq B$.
\end{defn}

In the case $k=2$, the group $\aautbis$ contains $\mathrm{Aut}(T_{d+1})$ as an open subgroup, where $T_{d+1}$ is the \textit{unrooted} regular tree with valance $d+1$. Moreover, our notions of elliptic elements and translations in $\aautbis$ agree with the notions of elliptic elements and hyperbolic elements in $\mathrm{Aut}(T_{d+1})$.

In a locally compact group $G$, an element $g\in G$ is said to be \textbf{periodic} if $\cgrp{g}$ is compact. Elliptic elements are easily seen to be periodic in $\aaut$, and translations are never periodic. We show the converses to these statements hold.

\begin{prop}[See Proposition~\ref{prop:char_1_elliptic}] 
	For $g\in \aaut$, the following are equivalent:
	\begin{enumerate}[(1)]
		\item $g$ is elliptic;
		\item $g$ is periodic;
		\item $g$ is not a translation.
	\end{enumerate}
\end{prop}

A subgroup of $\aaut$ is called \textbf{elliptic} if every element is elliptic. We characterize these subgroups in terms of the dynamics of the action on $\bord$; this result again extends the classical characterization of elliptic groups acting on trees \cite[Proposition 26]{Serre-trees}. Recall for a locally compact group $G$, a subgroup $H\leq G$ is \textbf{locally elliptic} if $\cgrp{F}$ is compact for all finite $F\subseteq H$. 

\begin{thm}[See Corollary~\ref{cor:elliptic_sgrp}] \label{thm-intro-tot-ell}
	For $H\leq \aaut$, the following are equivalent:
	\begin{enumerate}
		\item $H$ contains only elliptic elements;
		\item $H$ contains no translations;
		\item For every compactly generated subgroup $\Gamma \leq H$, there exists a set-wise $\Gamma$-invariant partition $\bord = B_1 \sqcup \ldots \sqcup B_n$ into finitely many balls such that $\gamma_{\rest B_i}: B_i \rightarrow \gamma(B_i)$ is a homothety for every $\gamma \in \Gamma$ and $1\leq i\leq n$;
		\item $H$ is locally elliptic. 
	\end{enumerate}
\end{thm}

\subsection{Commensurated subgroups}

Two subgroups $K,H \leq G$ of a group $G$ are \textbf{commensurable} if $K \cap H$ has finite index in both $K$ and $H$. A subgroup $H\leq G$ is \textbf{commensurated} in $G$ if $gHg^{-1}$ is commensurable with $H$ for every $g\in G$. The \textbf{commensurability class} of a subgroup $H$ is the set of subgroups $K$ such that $H$ and $K$ are commensurable.

Normal subgroups are trivial examples of commensurated subgroups, but commensurated subgroups need not be normal in general. Group actions on graphs are a natural source of commensurated subgroups: if $G$ acts on a connected locally finite graph, then vertex stabilizers must be commensurated in $G$. More generally, the point stabilizers of any subdegree finite permutation group are commensurated.

The study of commensurated subgroups is of particular interest when the ambient group has few normal subgroups, e.g.\ simple or just infinite groups. A striking example of such an exploration is the work of Y. Shalom and G. Willis, who classified commensurated subgroups for a large family of arithmetic groups \cite{SW_13}. As an additional example, the second named author classified commensurated subgroups for finitely generated just infinite branch groups \cite{W_B_16}.

In the present work, we classify the commensurated subgroups of certain groups of tree automorphisms and tree almost automorphisms. For simplicity, the following result is not stated here in full generality; see Theorem \ref{thm-commens-tree} for the full statement.

\begin{thm} \label{thm-aut(T)-intro}
Suppose that $\mc{T}$ is a biregular tree and $\Aut(\mc{T})^+$ is the subgroup of $\Aut(\mc{T})$ acting on $\mc{T}$ with two orbits of vertices. If $\Lambda$ is a commensurated subgroup of $\Aut(\mc{T})^+$, then either $\Lambda$ is finite, $\ol{\Lambda}$ is compact open, or $\Lambda = \Aut(\mc{T})^+$.
\end{thm}

While the results of Section \ref{sec-action-tree} for groups acting on trees follow from fairly classical arguments, the case of the group $\aaut$ is more complicated and requires new techniques and intermediate results, including Theorem \ref{thm-intro-tot-ell}. In addition, we use the following result of Caprace--Reid--Willis: \textit{If $G$ is a non-discrete compactly generated topologically simple \tdlc group which is abstractly simple, then every infinite compact commensurated subgroup is open.} See Theorem \ref{thm-CRW2-finite-open} below.

\begin{thm}[See Theorem \ref{thm-trichotomy}]\label{thm:aaut-intro}
If $\Lambda\leq \aaut$ is commensurated, then either $\Lambda$ is finite, $\ol{\Lambda}$ is compact and open, or $\Lambda=\aaut$. In particular, $\aaut$ admits exactly three commensurability classes of closed commensurated subgroups.
\end{thm}

It may happen that $\Lambda$ is commensurated but not closed (see Example \ref{ex-nonclosed-example}), so in the conclusions of Theorems \ref{thm-aut(T)-intro} and \ref{thm:aaut-intro}, passing to the closure $\ol{\Lambda}$ is necessary. We remark further that \tdlc groups admit a basis of neighborhoods at the identity of compact open subgroups, and any compact open subgroup must be commensurated in the ambient group. A non-discrete and non-compact \tdlc group therefore always admits at least three distinct commensurability classes of commensurated subgroups: the class of the trivial subgroup, the class of a compact open subgroup, and the class of the entire group. The conclusion of Theorem \ref{thm:aaut-intro} is thus optimal.


We then consider the three groups $F$, $T$, and $V$ introduced by R.\ Thompson. Recall that $T$ is the group of orientation-preserving homeomorphisms of the circle which are piecewise linear with power-of-two slopes and have only finitely many breakpoints all at dyadic rationals. The group $F$ is the stabilizer of $0$ in $T$, so $F$ acts on the interval $[0,1]$. For a detailed introduction to these groups, we refer the reader to \cite{CFP}. We will view $F$ and $T$ inside the group $V$, which is itself a subgroup of the almost automorphism group $\mathrm{AAut}(\mc{T}_{2,2})$; see Subsection \ref{subsec-aaut}.

\begin{thm}[See Section~\ref{sec-proofs-commens}] \label{thm-intro-thompson}
The following hold:
\begin{enumerate}
	\item The commensurated subgroups of Thompson's group $F$ are the normal subgroups of $F$. (These are the subgroups of $F$ that contain the derived subgroup of $F$.)
	\item Every proper commensurated subgroup of Thompson's group $T$ is finite.
	\item Every proper commensurated subgroup of Thompson's group $V$ is locally finite.
\end{enumerate}
\end{thm}

In particular, Thompson's group $T$ is an example of a finitely presented simple group with \textit{exactly two} commensurability classes of commensurated subgroups. To the best of our knowledge, no other group is known to have these properties. We mention that the examples of finitely presented groups which are torsion-free and simple, constructed by M.\ Burger and S.\ Mozes in \cite{BM-IHES-2} as lattices in product of trees, admit non-trivial proper commensurated subgroups, and these are always non-amenable; see the end of Section \ref{sec-action-tree}.

Our results also elucidate a difference between $T$ and $V$. The group $V$ does admit a commensurated subgroup $\Lambda \leq V$ which is infinite and of infinite index - namely the group of finitary automorphisms of $\mathcal{T}_{2,2}$; see for instance \cite[Example 6.7, Proposition 7.11]{LB14}. Hence, while $T$ has only two commensurability classes of commensurated subgroups, $V$ has at least three commensurability classes of commensurated subgroups: the trivial group, the group of finitary automorphisms of $\mathcal{T}_{2,2}$, and the entire group $V$.

\subsection{Applications} \label{subsec-appli}
Our first application shows there exist strong restrictions on the possible continuous morphisms from the group $\aaut$ to a \tdlc group. 

\begin{cor} \label{cor-intro-embed-aaut}
If $H$ is a \tdlc group and $\psi: \aaut \rightarrow H$ is a continuous morphism, then $\psi$ has closed image.
\end{cor}

We emphasize that the conclusion of Corollary \ref{cor-intro-embed-aaut} fails for certain groups of almost automorphisms of trees which may look similar to $\aaut$, notably the family of groups $\Daaut$ considered in \cite{CdM11, LB14, Sauer-Thu}. We refer the reader to the discussion following Corollary \ref{cor-aaut-closed-im} for details.

Since $\aaut$ is simple, Corollary \ref{cor-intro-embed-aaut} implies that any non-trivial continuous morphism $\psi:\aaut \rightarrow H$ is such that $\psi:\aaut\rightarrow \psi(\aaut)$ is an \textit{isomorphism of topological groups}.

%
%
Our results further impose restrictions on how Thompson's groups can embed into locally compact groups. Our first restrictions follow readily from Theorem~\ref{thm-intro-thompson}. The following results strongly contrast with the fact that Thompson's group $F$ and $T$ are respectively dense in the Polish groups $\mathrm{Homeo}^+([0,1])$ and $\mathrm{Homeo}^+(\mathbb{S}^1)$.

\begin{cor} \label{cor-intro-embed-F}
Any embedding of Thompson's group $F$ into a \tdlc group $H$ intersects trivially any compact open subgroup of $H$. In particular, if $F$ acts faithfully on a connected locally finite graph, then the action on the set of vertices is free.
\end{cor}

\begin{cor} \label{cor-intro-embed-T}
Any embedding of Thompson's group $T$ into a \tdlc group has discrete image. In particular, any non-trivial action of $T$ on a connected locally finite graph is proper.
\end{cor}

These results also establish another difference between the groups $F$,$T$, and $V$. The group $V$ admits non-proper vertex-transitive actions on connected locally finite graphs; for example, one can take a Cayley-Abels graph of the group $\mathrm{AAut}(\mathcal{T}_{2,2})$. 

We lastly consider how $T$ can appear as a lattice in a locally compact group.

\begin{thm}[See Theorem~\ref{thm:latticesT}]
Suppose that $G$ is a compactly generated locally compact group admitting $T$ as a lattice and denote by $R$ the locally elliptic radical of $G$. Then $R$ is compact, $G/R$ is a \tdlc group with a unique minimal non-trivial closed normal subgroup $H$, and $H$ satisfies the following properties:
\begin{enumerate}
	\item $H$ is a compactly generated topologically simple \tdlc group;
	\item $H$ is cocompact in $G/R$ and contains $T$ as a lattice.
\end{enumerate}  
\end{thm}

Whether or not $H$ can be non-discrete in the above theorem is an interesting open question. As there is a growing structure theory of simple \tdlc groups, see for example \cite{CRW13}, the above result may help to resolve this question.

\begin{ack}
We are grateful to Fran\c{c}ois Le Ma\^{i}tre for his many helpful comments. We also thank the anonymous referee for pointing out an improvement to the proof of Theorem~\ref{thm:non-comm}.

The first named author would like to thank the organizers of the conference \textit{Workshop on the extended family of Thompson's groups} held in St. Andrews, during which part of this project grew up. Portions of this work were also developed during a stay of both authors at the \textit{Mathematisches Forschungsinstitut Oberwolfach}; we thank the MFO for its hospitality.
\end{ack}

\section{Preliminaries}

\subsection{Notations for $\Trd$}
The tree $\Trd$ is the rooted tree where the root $r$ has degree $k$ and other vertices have degree $d+1$. The \textbf{level} of a vertex is its distance from the root. If $v$ is a vertex of level $n$, the neighbors of $v$ of level $n+1$ are the \textbf{children} of $v$. The \textbf{descendants} of $v$ are all vertices $w$ such that the geodesic from $w$ to the root contains $v$. We denote by $\Trd^v$ the subtree spanned by the descendants of $v$, and we call $\Trd^v$ the \textbf{rooted tree below} $v$. A subtree $L$ of $\Trd$ is \textbf{regular} if there is a vertex $v$ of level at least one such that $L = \Trd^v$. 

The boundary $\bord$ is the set of infinite sequences of vertices $(r=\xi_0,\xi_1,\ldots)$ such that $\xi_{n+1}$ is a child of $\xi_{n}$ for every $n \geq 0$. If $\xi,\xi' \in \bord$ are two boundary points, we denote by $N(\xi,\xi')$ the largest integer $n \geq 0$ such that $\xi_{n} = \xi_{n}'$. Endowed with the metric $\mathrm{dist}(\xi,\xi') := d^{-N(\xi,\xi')}$, the space $(\bord,\mathrm{dist})$ is a compact metric space homeomorphic to a Cantor set. The group of isometries of $(\bord,\mathrm{dist})$ is precisely the group $\aut$ of automorphisms of the rooted tree $\treedk$.

If $L$ is a regular subtree of $\Trd$ rooted at $v$, we will denote $\partial L$ the subset of $\bord$ consisting of sequences $(\xi_n)_{n\in \Nb}\in \bord$ such that $\xi_n = v$ for some $n \geq 1$. One may check that $\partial L$ is a proper ball of $\partial \Trd$ and conversely that every proper ball of $\partial \Trd$ is of the form $\partial L$ for some regular subtree $L$.

Let $\mc{P}$ be a partition of $\partial \Trd$. If $\mc{P}'$ is a partition refining $\mc{P}$, we write $\mc{P}'\sleq \mc{P}$. We say that the partition $\mc{P}$ is \textbf{regular} if $\mc{P}$ consists of proper balls; equivalently, $\mc{P}$ consists of sets of the form $\partial L$ for regular subtrees $L$. The \textbf{$n$-th spherical partition} of $\bord$ is
\[
\Sc_n:=\{\partial \Trd^v\mid d(r,v)=n\}.
\]

We make an easy observation that will often be used implicitly to manipulate regular partitions: \textit{For any metric balls $B_1,B_2\subseteq \bord$, either $B_1\cap B_2=\emptyset$, $B_1\subseteq B_2$, or $B_2\subseteq B_1$.}

\subsection{The group $\aaut$} \label{subsec-aaut}

We here give a brief account of the group $\aaut$; we direct the reader to \cite{CdM11,Ka99} for more detailed discussions. 

For metric spaces $(X,d_X)$ and $(Y,d_Y)$, a map $\psi:X\rightarrow Y$ is a \textbf{homothety} if there is $\lambda\in \mathbb{R}^{+}$ such that $d_X(x,x')=\lambda d_Y(\psi(x),\psi(x'))$ for all $x,x'\in X$. An \textbf{almost automorphism} of $\treedk$ is a homeomorphism $g \in \mathrm{Homeo}(\bord)$ such that there exists a regular partition $\bord = B_1 \sqcup \ldots \sqcup B_n$ for which $g_{\rest B_i}: B_i \rightarrow g(B_i)$ is a homothety for every $i$. 

\begin{defn} 
For $g \in \aaut$, a partition $\mc{P} = \left\{B_1,\ldots, B_n\right\}$ of $\bord$ into finitely many proper balls is \textbf{admissible} for $g$ if $g_{\rest B_i}: B_i \rightarrow g(B_i)$ is a homothety for every $i$. We say a partition $\mc{P}$ is admissible for a subgroup $K\leq\aaut$ if it is admissible for each element $k\in K$.
\end{defn} 

If $\mc{P}$ is admissible for $g\in \aaut$ and $\mc{P}'$ is a regular partition refining $\mc{P}$, then $\mc{P}'$ is admissible for $g$. We denote by $g.\mc{P}$ the partition obtained by applying $g$, i.e. $g.\mc{P}:=\{g(B)\mid B\in \mc{P}\}$. If $\mc{P}$ is admissible for $g$, then $g.\mc{P}$ is a regular partition admissible for $g^{-1}$.

\begin{defn} 
An element $g \in \aaut$ is \textbf{elliptic} if there exists a partition $\mc{P} = \left\{B_1,\ldots, B_n\right\}$ of $\bord$ that is admissible for $g$ and preserved by $g$ - i.e.\ $g.\mc{P}=\mc{P}$.
\end{defn}

For a regular partition $\mc{P}$ and $K\leq \aaut$, we define the \textbf{stabilizer} $\Stab_K(\mc{P})$ of $\mc{P}$ in $K$ to be the set of elements $k \in K$ such that $\mc{P}$ is admissible for $k$ and $k.\mc{P}=\mc{P}$. When $K=\aaut$, we simply write $\Stab(\mc{P})$. The group $\Stab(\mc{P})$ is a compact open subgroup of $\aaut$. Elements stabilizing a spherical partition admit a useful characterization.

\begin{lem}\label{lem:spherical_part}
	Let $g\in \aaut$. There is a spherical partition $\Sc_n$ such that $g\in \Stab(\Sc_n)$ if and only if for all partitions $\mc{P}$ admissible for $g$ and $B\in \mc{P}$, the restriction $g_{\rest B}:B\rightarrow g(B)$ is an isometry.
\end{lem}
\begin{proof}
	Suppose that $g\in \Stab(\Sc_n)$ and observe that $g_{\rest B}:B\rightarrow g(B)$ is an isometry for all $B\in \Sc_n$. Consider $\mc{P}$ an admissible partition for $g$. By taking a larger $n$ if needed, we may assume that $\Sc_n\sleq \mc{P}$. For each $C\in \mc{P}$, there is thus $B\in \Sc_n$ such that $B\subseteq C$. The element $g$ acts as an isometry on $B$, hence $g$ must act as an isometry on $C$, since $g$ acts as a homothety on $C$. We have thus verified the forward implication.
	
	Conversely, we may find a spherical partition $\Sc_n$ which is admissible for $g$. Since $g$ restricts to an isometry on each $B\in \Sc_n$, the image $g.\Sc_n$ is a regular partition of $\bord$ by balls of the same diameter. It follows that $g.\Sc_n=\Sc_n$.
\end{proof}

For a ball $B\subseteq \bord$ and $K\leq \aaut$, we define $\Stab_K(B)$ to be the set of elements $k\in K$ that setwise preserve $B$ and are such that $k_{\rest B}: B \rightarrow B$ is an isometry. The subscript $K$ will be suppressed when $K=\aaut$. The group $\Stab(B)$ is open; however, it is not compact when $B\subsetneq \bord$.

For the rest of the present article, we fix an embedding of $\treedk$ into the oriented plane such that each level $n$ of $\treedk$ lies on a horizontal line $y=r_n$ and $(r_n)_{n\in \Nb}$ is strictly monotone. This induces a total order on the boundary $\bord$. An element $g \in \aaut$ is \textbf{locally order preserving} if $g$ admits an admissible partition $\mc{P}$ such that $g_{\rest B_i}: B_i \rightarrow g(B_i)$ preserves the order for every $i$. The set of locally order preserving elements of $\aaut$ forms a subgroup isomorphic to the Higman-Thompson group $V_{d,k}$ \cite{Higman}. In this notation, Thompson's group $V$ is $V_{2,2}$.

We conclude our discussion by recalling that $\aaut$ is simple as an abstract group. 

\begin{thm}[Kapoudjian]\label{cor-aaut-simple}
	The group $\aaut$ is abstractly simple.
\end{thm}

The proof in \cite{Ka99} is given for the case $k=2$, but the argument readily adapts to arbitrary $k$ by using the simplicity of the commutator subgroup $[V_{d,k},V_{d,k}]$ of the Higman-Thompson groups $V_{d,k}$ and the simplicity of the group of type preserving automorphisms of a regular unrooted tree.

\subsection{Commensurated subgroups}

\begin{defn}
	For a group $G$ with subgroups $H$ and $K$, we say $H\leq G$ is \textbf{commensurated} by $K$ if $kHk^{-1}\cap H$ has finite index in $H$ for all $k\in K$. When $K=G$, we say $H$ is a \textbf{commensurated} subgroup of $G$

\end{defn}

We shall make use of the following equivalent definition of a commensurated subgroup: For $G$ a group, a subgroup $H\leq G$ is \textbf{commensurated} if the left $H$-action on $G/H$ has only finite orbits. We use this definition whenever convenient.

Our first two lemmas appear to be folklore; we include proofs for completeness.

\begin{lem}[Folklore]\label{lem:commensurated_and_normal}
	For $H$ a group, if $A$ and $B$ are commensurated subgroups of $H$ and $A$ normalizes $B$, then the subgroup $AB$ is commensurated in $H$.
\end{lem}
\begin{proof}
	Let $g\in G$ and consider $X:=BgAB\subseteq G/AB$. Since $B$ is commensurated, we find a finite set $F\subseteq G$ such that $BgB=FB$. We deduce that 
	\[
	X=BgAB=BgBA=FBA=FAB,
	\]
	so $X$ is a finite subset of $G/AB$. Since $A$ is commensurated, the set $AFA\subseteq G/A$ is finite, and thus, $AX$ is also finite. We see that $BAX=ABX=AX$ is a finite $AB$-orbit containing $gAB$. Since $g$ was arbitrary, we deduce that every left $AB$-orbit on $G/AB$ is finite, so $AB$ is commensurated. 
\end{proof}
\begin{lem}[Folklore]\label{lem:commensurated_closure}
	For $H$ a topological group, if $A$ is a commensurated subgroup, then $\ol{A}$ is also commensurated.
\end{lem}
\begin{proof}
	Taking $x\in H$, we may find a finite set $F\subseteq A$ such that $A\subseteq \bigcup FxAx^{-1}$. It follows $\ol{A}\subseteq \bigcup Fx\ol{A}x^{-1}$. We conclude that $\ol{A}\cap x\ol{A}x^{-1}$ has finite index in $\ol{A}$, hence $\ol{A}$ is commensurated.
\end{proof}

We shall need a more specialized version of Lemma~\ref{lem:commensurated_and_normal}. To prove the desired result, we make use of the following theorem independently proved by G.\ Schlichting \cite{schlich-comple} and G.M.\ Bergman and H.W.\ Lenstra \cite[Theorem 6]{BL89}. When $K$ satisfies the following theorem, we say that $K$ \textbf{boundedly commensurates} $H$.

\begin{thm} \label{thm:BL}
Let $G$ be a group with subgroups $H$ and $K$. Then the following are equivalent:
\begin{enumerate}[(1)]
\item The supremum $\sup_{k\in K}|H:H\cap kHk^{-1}|$ is finite.
\item There is $N\leq G$ commensurable with $H$ such that $K$ normalizes $N$.
\end{enumerate}
\end{thm}

In the setting of locally compact groups, the Baire category theorem allows us to obtain bounded commensuration relatively easily; the next lemma indeed holds for any Hausdorff topological group in which the Baire category theorem holds  - i.e.\ a \textit{Baire group}.

\begin{lem} \label{lem-cpt-boundedly} 
	Let $G$ be a locally compact group with $H \leq G$ a closed subgroup. If $K$ is a compact subgroup of $G$ that commensurates $H$, then $K$ boundedly commensurates $H$. 
\end{lem}

\begin{proof}
	For $n \geq 1$, consider the set
	\[
	\Omega_n:=\{u\in K\mid |H:H\cap uHu^{-1}|,|H:H\cap u^{-1}Hu|\sleq n\}.
	\] 
	Say $u_k\in \Omega_n$ is a net such that $u_k \rightarrow u\in K$ and fix $c_0,\dots,c_n \in H$. Since $u_k \in \Omega_n$, for every $k$ we may find $i_{k} \neq j_{k}$ such that $c_{i_{k}}c_{j_{k}}^{-1}\in u_k H u_k^{-1}$. By passing to a subnet, there exist $i \neq j$ such that $c_{i}c_{j}^{-1}\in u_k H u_k^{-1}$ for every $k$. Taking the limit as $k$ goes to infinity, we obtain that $c_{i}c_{j}^{-1}\in u H u^{-1}$. We have thus proved that the index $|H:H\cap uHu^{-1}|$ is at most $n$, and by applying the same argument to $u^{-1}$, we obtain that $|H:H\cap u^{-1}Hu|$ is also at most $n$. Therefore, $u \in \Omega_n$, and $\Omega_n$ is closed.  
	
	On the other hand, $K=\bigcup_{n\in \Nb}\Omega_n$ since $K$ commensurates $H$, so by the Baire category theorem, there is some $n$ such that $\Omega_n$ has non-empty interior. The compactness of $K$ implies that there is a finite set $F\subseteq K$ such that $K\subseteq \bigcup F\Omega_n$. We may find $m$ where $F\subseteq \Omega_m$, hence $K\subseteq \Omega_m\Omega_n$. An easy calculation now shows that $\Omega_m\Omega_n\subseteq \Omega_{mn}$, and thus, $K=\Omega_{mn}$.
\end{proof}

\begin{prop}\label{prop:contain U} 
Let $G$ be a \tdlc group with $H \leq G$ a closed commensurated subgroup. If $H$ is locally elliptic, then for every compact open subgroup $U \leq G$, there exists a commensurated, locally elliptic subgroup $H'$ such that $H'$ contains $U$ and $|H:H'\cap H|<\infty$.
\end{prop}

\begin{proof}
Since $U$ is compact and $H$ is closed and commensurated, Lemma \ref{lem-cpt-boundedly} implies that $U$ boundedly commensurates $H$. We are therefore in position to apply Theorem~\ref{thm:BL} to obtain a subgroup $L$ commensurable with $H$ such that $U \sleq N_G(L)$. 

We claim $H': = LU$ satisfies the proposition. Since $L$ is commensurable with $H$, the subgroup $L$ is locally elliptic, and thus, $H'$ is locally elliptic since being locally elliptic is stable by extension \cite{P_66}. Lemma \ref{lem:commensurated_and_normal} ensures $H'$ is commensurated, and $|H:H'\cap H|<\infty$, since $H'$ contains $H \cap L$ which has finite index in $H$.
\end{proof}

We finally require a deep result on commensurated subgroups of simple groups.

\begin{thm}[Caprace--Reid--Willis, {\cite[Theorem 3.9]{CRW_2_13}}] \label{thm-CRW2-finite-open}
	Suppose that $G$ is a compactly generated \tdlc group. If $G$ is abstractly simple, then every compact commensurated subgroup of $G$ is either finite or open.
\end{thm}

\section{Translations, elliptic elements, and elliptic subgroups}

A subgroup of $\aaut$ containing no translations will be called \textbf{elliptic}. This paragraph consists of a technical study of elliptic subgroups, which leads to a complete characterization of these.

For $g\in\aaut$, let $\mc{P}$ be an admissible partition and define
\[
\Omega_{\mc{P},g}:=\{Q\in g.\mc{P}\mid \exists \; B\in \mc{P}\text{ for which }B\subsetneq Q\}
\]
If $\Omega_{\mc{P},g}=\emptyset$, then each $Q\in g.\mc{P}$ is a subset of some $B\in \mc{P}$. Since $|\mc{P}|=|g.\mc{P}|$, we deduce that $\mc{P}=g.\mc{P}$, which implies that $g$ is elliptic. The set $\Omega_{\mc{P},g}$ is thus the failure of $\mc{P}=g.\mc{P}$. The goal of this subsection is to show that given an translation-free subgroup $K$, we can always refine a partition $\mc{P}$ to a partition $\mc{Q}$ with $\Omega_{\mc{Q},g}=\emptyset$ for every $g\in K$. 

We shall need a tool to refine partitions in a precise way. Let $Y$ be a set and $\mc{Z}\subseteq \ms{P}(Y)$, with $\ms{P}(Y)$ the power set of $Y$. We define $Z\in \min(\mc{Z})$ if and only if $Z\in \mc{Z}$ and for all $Z'\in \mc{Z}$ if $Z'\subseteq Z$, then $Z'=Z$. That is to say, $\min(\mc{Z})$ is the collection of $\subseteq$-minimal elements of $\mc{Z}$. For finite sets $\mc{Z}$, the set $\min(\mc{Z})$ is always non-empty.

Suppose that $A\subseteq \Sym(Y)$ is a collection of permutations. Each $a\in A$ acts on $\ms{P}(Y)$, and for $Z\in \ms{P}(Y)$, we denote this action by $a(Z)$. Given $R\subseteq Y$ and $\mc{X}\subseteq \ms{P}(Y)$, we now define the desired function:
\[
\Theta_A(R,\mc{X}):= \begin{cases}
\min\left\{ a(Z)\mid Z\in \mc{X}\text{ and }a(Z)\subseteq R\right\} &\mbox{if there is some } a(Z)\subseteq R; \\ 
\{R\} & \mbox{else. }
\end{cases} 
\]
The function $\Theta_A$ on input $(R,\mc{X})$ outputs either the set of minimal elements of the form $a(Z)$ with $Z\in \mc{X}$ and $a\in A$ contained in $R$ or the set $\{R\}$.

We now demonstrate how to use $\Theta_A$ to refine partitions. 

\begin{rmk} To gain intuition for the definitions and argumentation, it may be helpful to first consider the case of $K=\{1\}$ in the following. This case is enough to prove Proposition~\ref{prop:char_1_elliptic}.
\end{rmk}

\begin{lem}\label{lem:refine_K}
	Suppose that $\mc{P}$ is a regular partition admissible for a relatively compact group $K$. If $\mc{P}'\sleq \mc{P}$ is a regular partition, then $K.\mc{P'}:=\{k(B)\mid B\in \mc{P}'\text{ and }k\in K\}$ is a finite set and $\min(K.\mc{P'})$ is a regular partition.
\end{lem}
\begin{proof}
	For each $B\in \mc{P}'$, the subgroup $\Stab(B)$ is open, so since $K$ is relatively compact, $K/\Stab_K(B)$ is finite. The set $K.B:=\{k(B)\mid k\in K\}$ has a natural bijection with $K/\Stab_K(B)$, and thus, $K.B$ is finite. It now follows that $K.\mc{P}'$ is finite. 
	
	The set $K.\mc{P}'$ consists of balls since $\mc{P}'$ is admissible for each $k\in K$. That $\min(K.\mc{P}')$ is a partition is now an easy exercise.
\end{proof}

\begin{lem}\label{lem:delta_refine} Suppose that $g\in\aaut$, $K\leq \aaut$ is relatively compact, and $\mc{P}$ is an admissible partition for $g$ and $K$. If $\mc{P}'\leq \mc{P}$ is a regular partition, then the set
	\[
	\Delta(\mc{P}'):=\bigcup_{Q\in g.\mc{P'}}\Theta_K(Q,\mc{P}')
	\]
	enjoys the following properties:
	\begin{enumerate}
		\item $\Delta(\mc{P}')$ is a regular partition.
		\item $g^{-1}.\Delta(\mc{P}')$ is a regular partition refining $\mc{P}'$.
		\item Each $C\in g^{-1}.\Delta(\mc{P}')$ is either equal to some $B\in \mc{P'}$ or of the form $g^{-1}k(B)$ for some $k\in K$ and $B\in \mc{P}'$. 
	\end{enumerate}
\end{lem}

\begin{proof}
	The partition $\mc{P}'$ is again admissible for $g$, so $g.\mc{P}'$ is a regular partition. For the first claim, it suffices to show $\Theta_K(Q,\mc{P}')$ is a regular partition for each $Q\in g.\mc{P}'$. Fixing $Q\in g.\mc{P}'$, if $\Theta_K(Q,\mc{P}')=\{Q\}$, then there is nothing to prove. We thus suppose that there is $k(B)\subseteq Q$ for some $k\in K$ and $B\in \mc{P}'$. 
	
	Lemma~\ref{lem:refine_K} ensures that $\min(K.\mc{P'})$ is a regular partition. Since $Q$ is a ball and there is some $C\in \min(K.\mc{P}')$ contained in $Q$, it follows that $Q$ is partitioned by the elements of $\min(K.\mc{P}')$ contained in $Q$. That is to say, the elements of $\Theta_K(Q,\mc{P'})$ form a regular partition of $Q$. We deduce $(1)$. 
	
	By construction, $\Delta(\mc{P}')$ refines $g.\mc{P'}$, so since $g.\mc{P'}$ is admissible for $g^{-1}$, the partition $\Delta(\mc{P'})$ is admissible for $g^{-1}$. Hence, $g^{-1}.\Delta(\mc{P'})$ is a regular partition refining $\mc{P'}$, verifying $(2)$.
	
	For $(3)$, each $D\in \Delta(\mc{P'})$ is either equal to $g(B)$ for some $B\in \mc{P}'$ or of the form $k(B)$ for some $B\in \mc{P}'$ and $k\in K$. Each $C\in g^{-1}.\Delta(\mc{P'})$ is thus either equal to some $B\in \mc{P'}$ or of the form $g^{-1}k(B)$ for some $k\in K$ and $B\in \mc{P'}$. 
\end{proof}

Let us now prove the main technical theorem of this section. 

\begin{thm} \label{thm:main_ell_sgrps}
	Suppose that $g\in\aaut$, $K\leq \aaut$ is relatively compact, and $\mc{P}$ is admissible for $g$ and $K$. If $\grp{K,g}$ contains no translations, then there is a regular partition $\mc{Q}$ refining $\mc{P}$ such that $\grp{K,g}\leq \Stab(\mc{Q})$.
\end{thm}

\begin{proof}
	We use the function $\Delta$ defined in Lemma~\ref{lem:delta_refine} to build a refining sequence of regular partitions $(\mc{P}_i)_{i\in \Nb}$. Set $\mc{P}_0:=\mc{P}$. If we have defined $\mc{P}_n$, define 
	\[
	\mc{P}_{n+1}:=g^{-1}.\Delta(\mc{P}_n).
	\]
	Lemma~\ref{lem:delta_refine} ensures that $\mc{P}_{n+1}$ refines $\mc{P}_n$ and is a regular partition. Observe that each $B\in \mc{P}_{n+1}$ has the form $B=\gamma(Q)$ for some $\gamma \in \grp{K,g}$ and $Q\in \mc{P}$. An induction argument shows for such a $B$, the element $\gamma$ is such that $\gamma_{\rest Q}:Q\rightarrow B$ is a homothety. 
	
	Suppose for contradiction that the sequence of partitions $\mc{P}_n$ never stabilizes. Since the partitions are refining and finite, there is $B_0\in \mc{P}_0$ which is refined in infinitely many $\mc{P}_n$. Let $n_1$ be the next index such that $B_0$ is properly refined in $\mc{P}_{n_1}$. Since the partition $\mc{P}_{n_1}$ is finite, there is some $B_1\in \mc{P}_{n_1}$ such that $B_1\subsetneq B_0$ and $B_1$ is refined in infinitely many $\mc{P}_n$. We produce an infinite sequence $B_0\supsetneq B_1\supsetneq \dots$ by continuing in this fashion.
	
	Each $B_i$ has the form $\xi(Q)$ for some $\xi\in \grp{K,g}$ and $Q\in\mc{P}$. Since $\mc{P}$ is finite, we may find $B_j\subsetneq B_i$ where $B_{j}=\delta(Q)$ and $B_i=\xi(Q)$ for some $\xi,\delta\in \grp{K,g}$ and $Q\in \mc{P}$. The element $\delta$ acts as a homothey on $Q$. On the other hand, $\xi^{-1}$ acts as a homothety on $B_i$, so it acts as a homothety on $\delta(Q)=B_j$.  We thus deduce that $\gamma:=\xi^{-1}\delta$ acts as a homothety on $Q$ and that $\gamma(Q)\subsetneq Q$.  Hence, $\gamma$ is a translation, contradicting our assumption that $\grp{K,g}$ contains no translation.
	
	Fix $N\in \Nb$ such that our sequence of partitions stabilizes at $N$. Set $Q:=\mc{P}_N$ and observe that $\mc{Q}$ is admissible for $g$. We now compute
	\[
	\Omega_{\mc{Q},g}:=\{R\in g.\mc{Q}\mid \exists\; Q\in \mc{Q}\text{ for which }Q\subsetneq R\}.
	\]
	Suppose that $\Omega_{\mc{Q},g}$ is non-empty; say that $g(L)\in \Omega_{\mc{Q},g}$ and $Q\in \mc{Q}$ is such that $Q\subsetneq g(L)$. Forming $\Delta(\mc{Q})$, it follows that there is $Q'\in \mc{Q}$ and $k\in K$ such that $k(Q')\in \Delta(\mc{Q})$ and $k(Q')\subsetneq g(L)$. Hence, $g^{-1}k(Q')\in \mc{P}_{N+1}=\mc{P}_N=\mc{Q}$. On the other hand, $g^{-1}k(Q)\subsetneq L\in \mc{Q}$ contradicting that $\mc{Q}$ is a partition. We conclude that $\Omega_{\mc{Q},g}=\emptyset$, so $g\in \Stab(\mc{Q})$.
	
	Let us now take $k\in K$ and compute $\Omega_{\mc{Q},k}$.  Suppose that $\Omega_{\mc{Q},k}$ is non-empty; say that $k(L)\in \Omega_{\mc{Q},k}$ and $Q\in \mc{Q}$ is such that $Q\subsetneq k(L)$. Thus $k^{-1}(Q)\subsetneq L$, and since $g$ stabilizes $\mc{Q}$, we have that $k^{-1}(Q)\subsetneq g(g^{-1}(L))$ with $g^{-1}(L)\in \mc{Q}$.  Forming $\Delta(\mc{Q})$, it follows there is $Q'\in \mc{Q}$ and $k'\in K$ such that $k'(Q')\in \Delta(\mc{Q})$ and $k'(Q')\subsetneq g(g^{-1}(L))$, and thus, $g^{-1}k'(Q')\in \mc{P}_{N+1}=\mc{Q}$. On the other hand, $g^{-1}k'(Q')\subsetneq g^{-1}(L)\in \mc{Q}$ contradicting that $\mc{Q}$ is a partition. We conclude that $\Omega_{\mc{Q},k}=\emptyset$, so $k\in \Stab(\mc{Q})$. 
	
	We have established that $\grp{K,g}\leq \Stab(\mc{Q})$, verifying the theorem.
\end{proof}

Applying our technical theorem, we prove the promised results on elliptic elements and subgroups.

\begin{prop} \label{prop:char_1_elliptic} 
	For $g\in \aaut$, the following are equivalent:
	\begin{enumerate}[(1)]
		\item $g$ is elliptic;
		\item some power of $g$ is an isometry of $\bord$; i.e.\ there is $n \geq 1$ such that $g^n \in \aut$;
		\item $g$ is periodic;
		\item $g$ is not a translation.
	\end{enumerate}
\end{prop}

\begin{proof}
	The implication $(2)\Rightarrow (3)$ is immediate, and Theorem~\ref{thm:main_ell_sgrps} gives $(4)\Rightarrow (1)$.
	
	For $(1)\Rightarrow (2)$, suppose that $g$ is elliptic and let $m$ be the cardinality of an admissible partition $\mc{P}$ that is preserved by $g$. The power $g^{m!}$ stabilizes all the parts of $\mc{P}$, and therefore $g^{m!} \in \aut$.

	For $(3)\Rightarrow (4)$, suppose that $g$ is periodic and suppose for contradiction that $g$ is a translation. We may find a power $n\in \Zb$ and an admissible partition $\mc{P}$ for which there is $B \in \mc{P}$ with $g^n(B)\subsetneq B$. It then follows that $g^{in}(B)\subsetneq B$ for all $i \geq 1$. On the other hand, since the subgroup of $\aaut$ consisting of elements set-wise stabilizing the ball $B$ is open and $g^n$ is periodic, there is $m \geq 1$ such that $g^{mn}(B) = B$, which is absurd.
\end{proof}

\begin{cor}\label{cor:elliptic_sgrp} 
	For $H\leq \aaut$, the following are equivalent:
	\begin{enumerate}
		\item $H$ contains only elliptic elements;
		\item For every compactly generated subgroup $\Gamma\leq H$, there is a regular partition $\mc{R}$ such that $\Gamma\sleq\Stab(\mc{R})$;
		\item $H$ is locally elliptic. 
	\end{enumerate}
\end{cor}

\begin{proof}
	The implication $(2)\Rightarrow (3)$ is immediate. We deduce $(3)\Rightarrow (1)$ from Proposition \ref{prop:char_1_elliptic}, as every element in a locally elliptic subgroup is periodic.
	
	For $(1)\Rightarrow (2)$, there are finitely many $\gamma_1,\dots,\gamma_n\in \Gamma$ such that these along with $K:=\Gamma\cap \Aut(\Trd)$ generate $\Gamma$. Letting $\mc{P}$ be an admissible partition for $\gamma_1$, we see that $\mc{P}$ is admissible for $K$ and $\gamma_1$, and applying Proposition \ref{prop:char_1_elliptic}, $\grp{\gamma,K}$ contains no translations. Theorem~\ref{thm:main_ell_sgrps} now supplies a regular partition $\mc{Q}$ refining $\mc{P}$ such that $\grp{\gamma,K}\leq \Stab(\mc{Q})$. Repeating this argument, we deduce that $\Gamma\leq \Stab(\mc{R})$ for some regular partition $\mc{R}$.
\end{proof}

We stress an interesting consequence of Corollary~\ref{cor:elliptic_sgrp}: \textit{If $K\leq \aaut$ is compact, then there is a regular partition $\mc{P}$ such that $K\leq \Stab(\mc{P})$.}

\medskip

It is not hard to see from the definition of the Higman-Thompson group $V_{d,k}$ inside $\aaut$ that the elliptic elements of $V_{d,k}$ are precisely its elements of finite order. Corollary~\ref{cor:elliptic_sgrp} therefore implies the following result, which was proved in \cite[Theorem 3]{Rov-cfpsg}.

\begin{cor} \label{cor-torion-Vdk}
Every torsion subgroup of the Higman-Thompson group $V_{d,k}$ is locally finite.
\end{cor}

Corollary~\ref{cor:elliptic_sgrp} may also be applied to other interesting finitely generated subgroups of $\aaut$, for instance the finitely presented simple group containing the Grigorchuk group constructed in \cite{Rov-cfpsg}. We refer to the discussion following Question \ref{quest-commens-V} for more details. 

\section{Commensurated subgroups of groups acting on trees} \label{sec-action-tree}

We now begin our study of commensurated subgroups. In this section, we consider commensurated subgroups of groups acting on trees.

\begin{lem} \label{lem-comm-double}
Let $G$ be a group with a commensurated subgroup $\Lambda$. Then for every $g,h \in G$, there exists a finite index subgroup $\Lambda' \leq \Lambda$ such that $\left[\left[g,\Lambda'\right],h\right] \leq \Lambda$.
\end{lem}

\begin{proof}
Since the intersection of finitely many subgroups of finite index remains of finite index, there is a finite index subgroup $\Lambda' \leq \Lambda$ such that $x \Lambda' x^{-1} \leq \Lambda$ for every $x \in \left\{g,h,hg\right\}$. For $k \in G$, a simple computation yields 
\[
\left[\left[g,k\right],h\right] = (gkg^{-1}) k^{-1} (hkh^{-1}) (hg k^{-1} (hg)^{-1}).
\] 
Taking $k = \lambda \in \Lambda'$, this expressions gives a decomposition of $\left[\left[g,\lambda\right],h\right]$ as a product of four elements all of which belong to $\Lambda$. Therefore, $\left[\left[g,\lambda\right],h\right] \in \Lambda$, and the statement is proved.
\end{proof}

Let $T$ be a simplicial tree. If $G$ is a group acting on $T$, we say that the action is of \textbf{general type} if there exist two hyperbolic elements in $G$ without common endpoints. We say that the action is \textbf{minimal} if there is no proper invariant subtree.

\begin{prop} \label{prop-comm-Aut(T)}
Let $G \leq \Aut(T)$ whose action on $T$ is minimal and of general type and let $\Lambda \leq \Aut(T)$ be a subgroup that is commensurated by $G$. Then either
\begin{enumerate}
	\item $\Lambda$ stabilizes a vertex or an edge; or
	\item the action of $\Lambda$ on $T$ is of general type, and for every half-tree $T'$, $\Lambda$ contains the commutator subgroup of the pointwise fixator of $T'$ in $G$. 
\end{enumerate} 
\end{prop}

\begin{proof}
We assume that $\Lambda$ stabilizes neither a vertex nor an edge and prove that the second assertion holds. 

Let us first argue that the action of $\Lambda$ on $T$ must be of general type. If not, there is $\Omega \subset \partial T$ of cardinality one or two that is $\Lambda$-invariant such that $\Omega$ is the only $\Lambda$-invariant finite subset of $\partial T$. Since the action of $G$ on $T$ is of general type, we may find $g \in G$ such that $g(\Omega)$ is disjoint from $\Omega$. By assumption $g$ commensurates $\Lambda$, so there exists a finite index subgroup $\Lambda' \leq \Lambda$ such that $g \Lambda' g^{-1} \leq \Lambda$. Consequently $\Lambda$ has a finite index subgroup which preserves $g(\Omega)$, implying that the $\Lambda$-orbit of any element of $g(\Omega)$ in $\partial T$ is finite. This is a contradiction.

We now argue the action of $\Lambda$ on $T$ is minimal. It follows from the previous paragraph that $\Lambda$ contains hyperbolic elements and therefore admits a unique minimal invariant subtree $X$, which is the union of the axes of the hyperbolic elements of $\Lambda$. Let $\gamma \in \Lambda$ be a hyperbolic element and let $g \in G$. Since $G$ commensurates $\Lambda$, there is some $k \geq 1$ such that $g \gamma^k g^{-1} \in \Lambda$. The axis of the hyperbolic element $g \gamma^k g^{-1} \in \Lambda$ thus belongs to $X$. This axis is the image by $g$ of the axis of $\gamma$, so this proves that $X$ is in fact $G$-invariant. By minimality of the action of $G$ on $T$, we have $X=T$.

Let $T'$ be a half-tree of $T$. Since the action of $\Lambda$ on $T$ is minimal and of general type, $\Lambda$ must contain a hyperbolic element $\gamma$ whose axis is contained in $T'$. This is a classical fact; see for instance \cite[Lemma 4.3]{LB15}. Take $g,h \in G$ fixing pointwise $T'$. We show that the commutator $[g,h]$ belongs to $\Lambda$. The same argument as in the proof of \cite[Lemma 4.4]{LB15} shows that $[g,h] = [[g,\gamma],h]$. Moreover, this argument only depends on the axis of $\gamma$ and not on $\gamma$ itself, so we indeed have $[g,h] = [[g,\gamma^k],h]$ for every non-zero $k \in \mathbb{Z}$. Since $G$ commensurates $\Lambda$, Lemma \ref{lem-comm-double} supplies an integer $k \geq 1$, depending on $g,h$, such that $[[g,\gamma^k],h]$ belongs to $\Lambda$. The commutator $[g,h]$ therefore lies in $\Lambda$, and the proof is complete.
\end{proof}

We now obtain examples of non-discrete compactly generated simple \tdlc groups with exactly three commensurability classes of closed commensurated subgroups. We refer the reader to the original paper of Tits for the definition of the independence property \cite[Section 4.2]{Ti70}. For a subgroup $G \leq \Aut(T)$, we denote by $G^+$ the subgroup of $G$ generated by pointwise fixators of edges. Note that when $G$ is endowed with the compact-open topology, $G^+$ is an open subgroup of $G$. 

\begin{thm} \label{thm-commens-tree}
Suppose that $G$ is a closed subgroup of $\Aut(T)$ satisfying Tits' independence property and whose action on $T$ is minimal and of general type. Suppose further that the group $G^+$ acts cocompactly on $T$. If $\Lambda$ is a commensurated subgroup of $G^+$, then either $\Lambda$ is finite, $\ol{\Lambda}$ is compact open, or $\Lambda = G^+$.
\end{thm}

\begin{proof}
Since $G$ satisfies Tits' independence property, the group $G^+$ is abstractly simple according to Tits' theorem \cite[Th\'{e}or\`{e}me 4.5]{Ti70}. That $G^+$ acts cocompactly on $T$ implies that $G^+$ is additionally compactly generated; see for example \cite[Lemma 2.4]{CdM11}. 

If $\Lambda$ stabilizes a vertex or an edge, then $\Lambda$ is a relatively compact subgroup of $G^+$. We are then in position to apply Theorem \ref{thm-CRW2-finite-open}, which shows that $\Lambda$ is either finite or $\ol{\Lambda}$ is compact and open. 

In the case $\Lambda$ stabilizes neither a vertex nor an edge, Proposition \ref{prop-comm-Aut(T)} implies the subgroup $\Lambda$ must contain the subgroup $N$ of $G^+$ generated by the derived subgroups of fixators of half-trees in $G^+$. The subgroup $N$ is clearly normal in $G^+$. Furthermore, $N$ cannot be trivial. Indeed, otherwise the group $G^+$ would be locally abelian, which is impossible thanks to \cite[Theorem 2.2]{Will07}. We thus have have $N = G^+$, so in particular $\Lambda  = G^+$.
\end{proof}

If $T$ is a biregular tree, then the full automorphism group of $T$ satisfies the assumptions of Theorem \ref{thm-commens-tree}, and therefore any proper commensurated subgroup of $\mathrm{Aut}(T)^+$ is either finite or has compact open closure. 

The following example shows that it is not possible to conclude in the second case of Theorem~\ref{thm-commens-tree} that $\Lambda$ itself is compact open. We remark that the same construction holds in the group $\aaut$, so $\aaut$ also admits a multitude of non closed commensurated subgroups. 

\begin{ex} \label{ex-nonclosed-example}
Fix a vertex $v_0$ in $T$. Let $\Aut(T)_{(v_0)}$ be the stabilizer of $v_0$ in $\Aut(T)$ and define the homomorphism 
\[ 
\pi: \mathrm{Aut}(T)_{(v_0)} \rightarrow \prod_{n \geq 1} \left\{\pm 1\right\} = :K 
\]
by taking the signature of the permutation induced on every sphere around $v_0$.

For an ultrafilter $\omega$ on $\mathbb{N}_{>0}$, the set $K_{\omega}$ of sequences $(x_n)_{n\geq 1} \in K$ such that $\omega$-almost surely $x_n = 1$ is a subgroup of index two of $K$. Since $\pi$ is onto, the preimage $P_{\omega}$ of $K_{\omega}$ in $\mathrm{Aut}(T)_{(v_0)}$ is of index two in $\mathrm{Aut}(T)_{(v_0)}$. The group $P_{\omega}$ is then commensurable with a compact open subgroup of $\mathrm{Aut}(T)^+$, and thus, it must be commensurated in $\mathrm{Aut}(T)^+$. However, one easily checks that $P_{\omega}$ is closed in $\mathrm{Aut}(T)^+$ only if the ultrafilter $\omega$ is principal, and consequently there exist many non-closed commensurated subgroups in $\mathrm{Aut}(T)^+$.
\end{ex}

An important class of groups satisfying Tits' independence property is the collection of groups $U(F)$ introduced by Burger and Mozes, consisting of automorphisms of a $d$-regular tree whose local action is prescribed by a permutation group $F$. For a formal definition and basic properties of these groups, we refer the reader to \cite{BM-IHES}. We only mention that the permutation groups $F$ which are transitive and generated by their point stabilizers are exactly those for which the group $U(F)^+$ is the subgroup of index two of $U(F)$ preserving the type of vertices. 

We deduce from Theorem \ref{thm-commens-tree} the following result for commensurated subgroups of $U(F)$.

\begin{cor} \label{cor-comens-u(f)}
If $F$ is transitive permutation group generated by its point stabilizers, then every proper commensurated subgroup of $U(F)^+$ is either finite or has compact open closure.
\end{cor}

In the case $F$ is two transitive and the point stabilizers are perfect, \cite[Theorem 8.6]{LW15} shows we need not pass to closures - i.e.\ every proper commensurated subgroup of $U(F)^+$ is either finite or compact and open. An example of such a group is $U(A_6)^+$ where $A_6$ is a permutation group in the natural manner.

We conclude this section with an observation about groups acting on products of trees. If $T_1$ and $T_2$ are two simplicial trees, we will denote by $\mathrm{pr}_1$ and $\mathrm{pr}_2$ the projections from $\mathrm{Aut}(T_1) \times \mathrm{Aut}(T_2)$ onto, respectively, $\mathrm{Aut}(T_1)$ and $\mathrm{Aut}(T_2)$. 

\begin{prop} \label{prop-commens-product-trees}
Let $T_1,T_2$ be locally finite trees and $\Gamma \leq \mathrm{Aut}(T_1) \times \mathrm{Aut}(T_2)$ be a discrete subgroup such that the action of $\Gamma$ on $T_i$ is minimal and of general type for $i=1,2$. Assume further that the projections $\mathrm{pr}_1(\Gamma)$ and $\mathrm{pr}_2(\Gamma)$ are non-discrete. Then the following hold:
\begin{enumerate}
	\item $\Gamma$ has no infinite amenable commensurated subgroups, and
	\item $\Gamma$ admits infinite and infinite index commensurated subgroups.
\end{enumerate}
\end{prop}

\begin{proof}
Assume that $\Lambda$ is an amenable commensurated subgroup of $\Gamma$. The actions of $\Lambda$ on $T_1$ and $T_2$ cannot be of general type, since otherwise $\Lambda$ would contain non-abelian free subgroups. Since $\Lambda$ is commensurated in $\Gamma$, it follows from Proposition \ref{prop-comm-Aut(T)} that $\Lambda$ must stabilize a vertex or an edge in each factor. Since $\Gamma$ is discrete in $\mathrm{Aut}(T_1) \times \mathrm{Aut}(T_2)$, this implies that $\Lambda$ must be finite.

For $(2)$, choose a vertex $v$ of $T_1$ and denote by $\Lambda$ the stabilizer of $v$ in $\Gamma$ for the projection action of $\Gamma$ on $T_1$. Since $T_1$ is locally finite, $\Lambda$ is commensurated in $\Gamma$. The subgroup $\Lambda$ is not finite because otherwise $\mathrm{pr}_1(\Gamma)$ would be discrete in $\mathrm{Aut}(T_1)$. On the other hand, $\Lambda$ is not of finite index in $\Gamma$, because this would imply that the $\Gamma$-orbit of $v$ is finite, contradicting the fact that the action of $\Gamma$ on $T_1$ is of general type. 
\end{proof}

Proposition \ref{prop-commens-product-trees} applies in particular to the finitely presented simple groups constructed by Burger and Mozes in \cite{BM-IHES-2}.

\begin{rmk}
The (non-)existence of infinite amenable commensurated subgroups is a property that naturally appears in the study of lattice envelopes of countable groups; see \cite{BFS15}.
\end{rmk}

\section{Technical results for almost automorphism groups}

For our results for commensurated subgroups in almost automorphism groups, we require several technical theorems. This section establishes these.

\subsection{Commensurated subgroups with a translation}

\begin{lem} \label{lem-trans-2balls}
If $\gamma \in \aaut$ is a translation, then there exist disjoint balls $B_1$ and $B_2$ of $\bord$ and $n \in \mathbb{Z}$ such that $\gamma^n(B_1) \lneq B_2$ and $\gamma^n(B_2) \lneq B_2$.
\end{lem}

\begin{proof}
By assumption, there is a ball $B$ and $n \in \mathbb{Z}$ such that $\gamma^n(B) \lneq B$. Take $B_2 := \gamma^n(B)$ and $B_1$ any proper ball of $B$ disjoint from $B_2$. The verification that these balls satisfy the conclusion is immediate. 
\end{proof}

In the next statement and its proof, we adopt the following notation: if $G \leq \aaut$ and $B$ is a ball of $\bord$, we denote by $G_B$ the subgroup of $G$ acting trivially outside of $B$.  

The proof of the next result is in the same spirit as the proof of Proposition \ref{prop-comm-Aut(T)}; that is to say, it uses the classical \enquote{double commutator} trick.

\begin{prop} \label{prop-tran-branch}
Suppose that $\Lambda \leq \aaut$ contains a translation. If $G\leq \aaut$ commensurates $\Lambda$, then there exists a proper ball $B$ such that $[G_B,G_B] \leq \Lambda$. If additionally $L\leq [G_B,G_B]$ has no proper finite index subgroup, then $\Lambda$ contains the group generated by all $G$-conjugates of $L$.
\end{prop}

\begin{proof}
Fix $\gamma \in \Lambda$ a translation. Thanks to Lemma \ref{lem-trans-2balls}, there exist disjoint balls $B_1, B_2$ and $n \in \mathbb{Z}$ such that $\gamma^n(B_1) \lneq B_2$ and $\gamma^n(B_2) \lneq B_2$. The obvious induction shows that we indeed have $\gamma^{kn}(B_1) \lneq B_2$ and $\gamma^{kn}(B_2) \lneq B_2$ for every $k \geq 1$.

Fix some $k \geq 1$. Taking $x\in G_{B_1}$, we consider the commutator $u_k := [x,\gamma^{kn}]$. On the complement of $B_1$, $x$ acts trivially, so $u_k$ acts trivially outside of $B_1 \cup \gamma^{kn}(B_1)$. The element $u_k$ additionally coincides with $x$ on $B_1$ and with $\gamma^{kn} x^{-1} \gamma^{-kn}$ on $\gamma^{kn}(B_1)$. 

For an arbitrary $y \in G_{B_1}$, we consider the element 
\[
v_k: = [u_k ,y] = [[x,\gamma^{kn}],y].
\]
As with $u_k$, the element $v_k$ acts trivially outside $B_1 \cup \gamma^{kn}(B_1)$. Furthermore, $v_k$ is also trivial on $\gamma^{kn}(B_1)$, because $y$ acts trivially on this ball. We thus deduce that $v_k$ acts trivially outside $B_1$ and acts on $B_1$ as $[x,y]$. Since the element $[x,y]$ is supported in $B_1$, we have $v_k = [x,y]$.

To conclude the proof, we apply Lemma \ref{lem-comm-double} to obtain some $k \geq 1$ such that $v_k = [[x,\gamma^{kn}],y]$ belongs to $\Lambda$. The previous paragraph now implies that $[x,y]$ is in $\Lambda$. Since $x$ and $y$ are arbitrary elements of $G_{B_1}$, we have proved that $\Lambda$ contains the commutator subgroup of $G_{B_1}$. 

If $L$ has no proper finite index subgroup, then since $\Lambda$ is commensurated by $G$, one must have $L \leq \Lambda \cap g \Lambda g^{-1}$ for every $g$ in $G$. The group $\Lambda$ thus contains all the conjugates of $L$ by elements of $G$, and so it contains the group generated by the conjugates.
\end{proof}

In the following result, $F_{d,k}$ and $T_{d,k}$ denote K.\ Brown's generalizations of Thompson's groups $F$ and $T$; see \cite{Brown}. The groups $F_{d,k}$ and $T_{d,k}$ are viewed as subgroups of $V_{d,k}$, so in particular they sit inside $\aaut$.

\begin{thm} \label{thm:no_trans}
Let $G$ denote one of following groups:
\begin{enumerate}[(i)]
\item the commutator subgroup $F_{d,k}'$ of the group $F_{d,k}$;
\item the group $T_{d,k}$ for $d,k$ with $\gcd(k,d-1)=1$;
\item the commutator subgroup $V_{d,k}'$ of the group $V_{d,k}$; or
\item the almost automorphism group $\aaut$;
\end{enumerate}
If $\Lambda \leq \aaut$ is commensurated by $G$ and contains a translation, then $\Lambda$ contains $G$.
\end{thm}

\begin{proof}
In each of these cases, the group $G$ is a simple group; see \cite{Brown} for $(i)-(iii)$ and Theorem~\ref{cor-aaut-simple} for $\aaut$. 

Since $\Lambda$ contains a translation, Proposition \ref{prop-tran-branch} ensures the existence of a proper ball $B$ of $\bord$ such that $[G_B,G_B] \leq \Lambda$. It is easy to see that the subgroup $G_B$ also contains a non-trivial simple group: For $G=F'_{d,k}$ or $G=T_{d,k}$, the group $G_B$ contains $F_{d,d}'$. In the case of $G=V_{d,k}'$ or $G=\aaut$, we have that $G_B$ contains $V_{d,d}'$ or $\mathrm{AAut}(\mc{T}_{d,d})$. 

The second statement of Proposition \ref{prop-tran-branch} now implies the group $\Lambda$ must contain the normal closure of some non-trivial subgroup of $G_B$ in $G$. As $G$ itself is simple, we conclude that $G\leq \Lambda$, completing the proof.
\end{proof}

\subsection{Elliptic commensurated subgroups} \label{sec-commensurated-LE}

We now consider commensurated subgroups $D\sleq \aaut$ which contain no translations. Corollary~\ref{cor:elliptic_sgrp} ensures the subgroup $D$ is indeed locally elliptic, so we will study commensurated locally elliptic subgroups. In view of Proposition~\ref{prop:contain U}, a commensurated locally elliptic subgroup can be extended by a compact open subgroup and remain both locally elliptic and commensurated. We thus consider locally elliptic subgroups that contain $\Aut(\Trd)$.

\subsubsection{Preliminaries} 
Recall that the $n$-th \textbf{spherical partition} of $\bord$ is 
\[
\mc{S}_n:=\{\partial\Trd^v\mid d(r,v)=n\}
\]
where $r$ is the root of $\Trd$. 

\begin{lem}\label{lem:spheres}
Suppose that $H\leq \aaut$ is locally elliptic and that $\Stab_H(\mc{S}_m)$ acts transitively on $\mc{S}_m$ for each $m\geq 1$. For each $h\in H$, there is then $n\geq 1$ such that $h\in \Stab(\Sc_n)$.
\end{lem}

\begin{proof}
Since $h$ is periodic, Proposition~\ref{prop:char_1_elliptic} ensures $h\in \Stab(\mc{P})$ for some regular partition $\mc{P}$. Suppose for contradiction that we cannot choose $\mc{P}$ to be a spherical partition. Lemma~\ref{lem:spherical_part} implies there is an admissible partition $\mc{Q}$ for $h$ and $B\in \mc{Q}$ such that $h_{\rest B}:B\rightarrow h(B)$ is not an isometry. The homothety $h_{\rest B}$ is thus such that the ball $h(B)$ has either strictly smaller or strictly larger diameter than that of $B$. As the cases are similar, let us suppose that the former holds.

Let $\Sc_m$ be a spherical partition refining $\mc{Q}$. Taking $C\in \Sc_m$ such that $C\subseteq B$, the ball $h(C)$ has strictly smaller diameter than $C$. We may then find $D\in \Sc_m$ such that $h(C)\subsetneq D$.  On the other hand, $\Stab_H(\Sc_m)$ acts transitively on $\Sc_m$, so there is $g\in \Stab_H(\Sc_m)$ such that $g(D)=C$. Forming the element $gh$, we see that $gh(C)\subsetneq C$, and thus, $gh$ is a translation, contradicting our assumption on $H$.
\end{proof}

Via Lemmas~\ref{lem:spheres} and \ref{lem:spherical_part}, each element $g$ of a locally elliptic subgroup containing $\Aut(\Trd)$ acts as an isometry on the parts of any admissible partition for $g$. This allows us to isolate a property incompatible with commensuration.

\begin{defn}
Suppose that $B$, $U$, and $W$ are balls in $\bord$. We say $h\in \aaut$ \textbf{breaks the tree below $B$} for $(U,W)$ if $\diam(U)=\diam(W)$ with $U,W\subseteq B$, the element $h$ fixes pointwise $W$, and $h$ acts as a homothety on $U$ with $B\cap h(U)=\emptyset$. We call a sequence of triples $((h_i,U_i,W_i))_{i\in I}$ a \textbf{breaking sequence} for $B$ if $h_i$ breaks the tree below $B$ for $(U_i,W_i)$ and $U_{i+1},W_{i+1}\subsetneq W_i$. 
\end{defn}
\begin{figure}[h]\label{fig:breaking_element}
    \centering
    \includegraphics[width=0.6\textwidth]{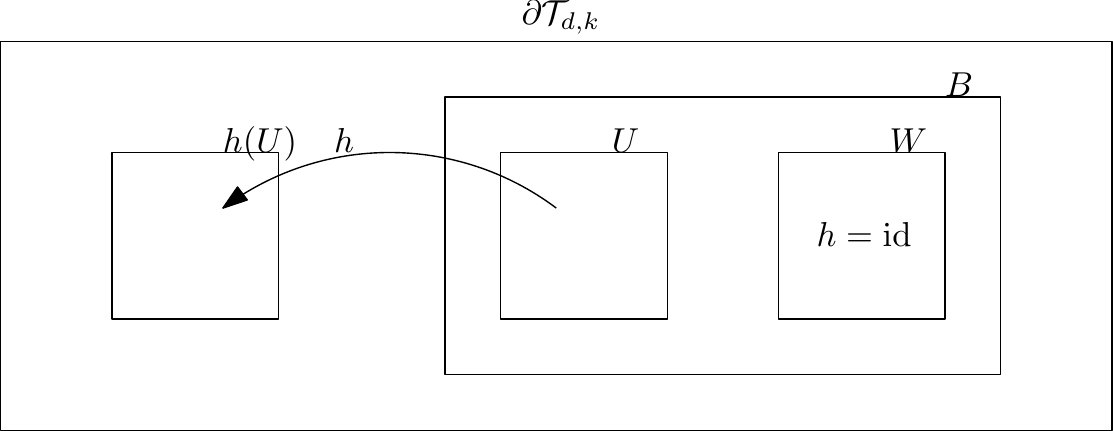}
    \caption{An element $h\in \aaut $ breaking the tree below $B$ for $U$ and $W$.}
 \end{figure}

\begin{lem}\label{lem:breaking_sequence}
Suppose that $H\leq \aaut$ is locally elliptic and  $\Stab_H(\mc{S}_m)$ acts transitively on $\mc{S}_m$ for each $m\geq 1$. If there is an infinite breaking sequence $((h_i,U_i,W_i))_{i\in \Nb}$ for a ball $B$ with $h_i\in H$ for all $i$, then $H$ is not commensurated.
\end{lem}

\begin{proof}
For each $k\in \Nb$, the sequence $((h_i,U_i,W_i))_{i>k}$ is a breaking sequence for the tree below $W_k$. By replacing $B$ with $W_2$ if necessary, we may assume $\diam(B) \leq d^{-2}$. 

Fix $m \geq 2$ such that the spherical partition $\Sc_m$ contains $B$. Lemmas~\ref{lem:spheres} and \ref{lem:spherical_part} imply that $h_i(U_i)$ is a ball of diameter $\diam(U_i)$ for each $i\in \Nb$. There is thus some $B_i\in \Sc_m$ such that $h_i(U_i)\subsetneq B_i$. Since $\Sc_m$ is finite, we may find $B'\in \Sc_m$ and an infinite subsequence  $(h_{i_j},U_{i_j},W_{i_j})$ such that $h_{i_j}(U_{i_j})\subsetneq B'$ for all $j$. By replacing $((h_i,U_i,W_i))_{i\in \Nb}$ with this subsequence, we assume that each term of $((h_i,U_i,W_i))_{i\in \Nb}$ is such that $h_i(U_i)\subsetneq B'$. 

Since $\diam(B')\leq d^{-2}$, there exists $g\in \aaut$ such that $\Sc_m$ is admissible for $g$, $g$ is the identity on $B$, and $g$ does not act like an isometry on $B'$. The element $gh_ig^{-1}$ does not act as an isometry on $g(U_i)=U_i$, so Lemma~\ref{lem:spheres} ensures that $gh_ig^{-1} \notin H$. For $j>i$, we further have that $U_j \subsetneq W_i$ and that $h_i$ fixes $W_i$ pointwise, and thus, $gh_jh_i^{-1}g^{-1}$ does not act as an isometry on $U_j$. Applying Lemma~\ref{lem:spheres} again, we deduce that $gh_jh_i^{-1}g^{-1} \notin H$.

Each $gh_ig^{-1}$ thus gives a distinct right coset of $H$. Hence, 
\[
|gHg^{-1}:gHg^{-1}\cap H|=|H:H\cap g^{-1}Hg|
\]
is infinite, showing that $H$ is not commensurated.
\end{proof}

We will also make use of a weak version of a breaking sequence. Set $K_{i+1}:=\Stab(\Sc_{i+1})\setminus \Stab(\Sc_i)$. Recalling that we fixed an embedding of $\Trd$ into the oriented plane, the spheres of $\Trd$ admit a linear ordering. For each spherical partition $\Sc_n$, let $R_n$ and $S_n$ be the parts corresponding to the penultimate and final elements of the ordering of the $n$-sphere of $\Trd$, respectively. Observe that $R_{n+1},S_{n+1}\subset S_n$.

\begin{lem}\label{lem:breaking element}
Suppose that $H$ is a non-compact locally elliptic subgroup which contains $\Aut(\Trd)$. There is then an infinite $I\subseteq \Nb$ where for each $n\in I$, there is $z\in H\cap K_{n}$ such that $z$ fixes pointwise $S_n$ and $z(R_n)\cap S_{n-1}=\emptyset$.
\end{lem}

\begin{proof}
Via Lemma~\ref{lem:spheres}, each element of $H$ is an element of some $\Stab(\Sc_n)$. Since $H$ is non-compact, there can be no upper bound on the least $n$ such that $h\in\Stab(\Sc_n)$ as $h$ ranges over $H$. It follows there is an infinite set $I\subseteq \Nb$ such that $H\cap K_n\neq \emptyset$ for all $n\in I$.

Fix $n\in I$ and take $z\in H\cap K_{n}$. Since $z$ is not an element of $\Stab(\Sc_{n-1})$, we may find $C\in \Sc_{n-1}$ and $B\neq D$ in $\Sc_n$ with $B, D\subset C$ such that $z(B)$ and $z(D)$ lie in different parts of $\Sc_{n-1}$. Multiplying by an appropriate element of $\Aut(\Trd)$, we may take $C=S_{n-1}$ and $B=S_{n}$. We may further find $u\in \Aut(\Trd)$ such that $uz$ fixes $S_{n}$ pointwise. The element $uz$ is thus such that $uz$ is an element of $K_{n}$, it fixes $S_{n}$ pointwise, and there is $D\subseteq S_{n-1}$ with $D\in \Sc_n$ such that $uz(D)\cap S_{n-1}=\emptyset$. The group $\Aut(\Trd)$ acts as the full symmetric group on the parts of $\Sc_n$ below $S_{n-1}$, so there is $v\in \Aut(\Trd)$ such that $v(R_n)=D$ and $v$ fixes $S_n$ pointwise. The element $uzv$ now satisfies $(2)$.
\end{proof}

\begin{defn}
For $n \geq 1$ and $H\leq \aaut$, a triple $(z,R_n,S_n)$ where $z\in H \cap K_{n}$ fixes pointwise $S_n$ and $z(R_n)\cap S_{n-1}=\emptyset$ is called a \textbf{weakly breaking triple} for $H$.
\end{defn}
Each $z\in K_{i+1}$ acts as a homothety on $\Sc_{i+1}$ by the definition of $\Stab(\Sc_{i+1})$. The almost automorphism $z$ in a weakly breaking triple $(z,R_{i+1},S_{i+1})$ therefore acts as a homothety on $R_{i+1}$.

Given a locally elliptic non-compact subgroup $H$ containing $\Aut(\Trd)$, Lemma~\ref{lem:breaking element} supplies a sequence of weakly breaking triples $(z_n,R_n,S_n)$ where $z_n \in H$ for all $n \in I$. We stress that there is no a priori reason for this sequence or any subsequence to be a breaking sequence.

Let us first observe an easy consequence of commensuration. The next lemma does not rely on any special property of weakly breaking triples. It uses only that $|H:H\cap gHg^{-1}|$ is finite.

\begin{lem}\label{lem:finite_depth}
Let $H\leq\aaut$ and take $g\in\aaut$ which commensurates $H$. There is then a finite set $Z$ of weakly breaking triples for $H$ and $N\geq 2$ such that for all weakly breaking triples $(h,R_k,S_k)$ with $k\geq N$ and $h\in H$, there is $(z,R_i,S_i)\in Z$ for which $hz^{-1}\in H\cap gHg^{-1}$.
\end{lem}

\begin{proof}
We recursively build the set $Z$ of weakly breaking triples. Take $n_1\in \Nb\setminus[0,2]$ least such that $(z_{1},R_{n_1},S_{n_1})$ is a weakly breaking triple with $z_{1}\in H$. Suppose that we have built our sequence up to $k$. If every weakly breaking triple $(y,R_m,S_m)$ for $m\in \Nb\setminus [0,n_k]$ and $y\in H$ is such that $yz_{i}^{-1}\in H\cap gHg^{-1}$ for some $1\leq i \leq k$, we stop. Else, find $n_{k+1}$ least in $\Nb\setminus[0,n_k]$ such that there is a weakly breaking triple $(z_{k+1},R_{n_{k+1}},S_{n_{k+1}})$ with $z_{k+1}\in H$ and $z_{k+1}z_{i}^{-1}\notin H\cap gHg^{-1}$ for all $1\leq i\leq k$.
	
	Our construction produces a set $Z$ of weakly breaking triples, and each group element $z$ appearing in a triple in $Z$ gives a distinct right coset of $H\cap gHg^{-1}$ in $H$. Since $|H:H\cap gHg^{-1}|$ is finite, the set $Z$ is finite, and thus, our construction procedure halts. That is to say, there is $N\geq 2$ such that for any $k\geq N$ and any weakly breaking triple $(h,R_k,S_k)$ with $h\in H$, we have that $hz^{-1}\in H\cap gHg^{-1}$ for some $(z,R_i,S_i)\in Z$.
\end{proof}

We now emphasize an important fact about weakly breaking triples: If $(a,R_i,S_i)$ and $(b,R_j,S_j)$ are weakly breaking triples with $j>i$, then $(ba,R_j,S_j)$ is a weakly breaking triple. In particular, in the setting of Lemma~\ref{lem:finite_depth}, if $(h,R_k,S_k)$ is weakly breaking with $k\geq N$ and $(z,R_i,S_i)\in Z$, then the triple $(hz^{-1},R_k,S_k)$ is weakly breaking.

\subsubsection{Main theorem}
We are now ready to prove the main theorem of this section. We shall use a general fact, which we leave as an exercise for the reader: There is an element $g\in \aaut$ admitting $\Sc_2$ as an admissible partition such that $g(S_n)=S_{n+1}$ and $g(R_n)=R_{n+1}$ for all $n\geq 2$. Such an element is called a \textbf{translation down the rightmost branch}.

\begin{thm}\label{thm:non-comm}
If $H\leq \aaut$ is locally elliptic, commensurated, and $\Aut(\Trd)\leq H$, then $H$ is compact.
\end{thm}
\begin{proof}
Suppose toward a contradiction that $H$ is non-compact. Fix $g\in G$ a translation down the rightmost branch and let $Z$ and $N\geq 2$ be as given by Lemma~\ref{lem:finite_depth}. We will abuse notation and consider $Z\subseteq H$. In view of Lemma~\ref{lem:breaking element}, for each $k\geq N$ there are infinitely many weakly breaking triples of the form $(y,R_l,S_l)$ with $l\geq k$ and $y\in H$;  we will often implicitly use this fact.

For balls $B$ and $D$ in $\bord$, define $\delta(B,D)$ to be the greatest $l$ such that $B,D\subseteq S_l$; this function is only partially defined. The partial function $\delta$ allows us to find weakly breaking triples which cohere as follows.
\begin{claim*}
For all $l>N$, there is $k>l$ and $w\in H$ such that $(w,R_k,S_k)$ is a weakly breaking triple for $H$ with $w(R_k)\subset R_l$. 
\end{claim*}

\begin{proof}[Proof of claim.]
In view of Lemma~\ref{lem:breaking_sequence}, there is $k>l$ and a weakly breaking triple $(h, R_k,S_k)$ with $h\in H$ such that $h(R_k)\subseteq S_{l-1}$; else we can produce an infinite sequence breaking for $S_{l-1}$. Take $k\geq l$ to be least for which there is such a weakly breaking triple. Note that $k>l$ and that $\delta(h(R_k),R_k)\geq l-1$. 

Since $k>N$, there is $z\in Z$ such that $hz^{-1}\in H\cap gHg^{-1}$, and by construction, $(hz^{-1},R_k,S_k)$ is a weakly breaking triple. Taking $x\in H$ such that $hz^{-1}=gxg^{-1}$, that $g$ is a shift down the rightmost branch implies $(x,R_{k-1},S_{k-1})$ is a weakly breaking triple. The minimality of $k$ ensures that $x(R_{k-1}) \nsubseteq S_{l-1}$, hence 
\[
\delta(x(R_{k-1}),R_{k-1})=\delta(h(R_k),R_k)-1\leq l-1.
\]
We deduce that $\delta(h(R_k),R_k)=l-1$. 

The ball $h(R_k)$ is thus contained in some $B\in \Sc_l$ such that $B\subseteq S_{l-1}$ and $B\neq S_l$. Since $\Aut(\Trd)$ acts as the full symmetric group on the parts of $\Sc_l$ contained in $S_{l-1}$, there is $t\in \Aut(\Trd)$ such that $t(B)=R_l$ and $t$ fixes pointwise $S_l$. The triple $(th,R_{k},S_{k})$ therefore satisfies the claim.  
\end{proof}

Applying the claim repeatedly, we can find an infinite sequence $(n_i)_{i\in \Nb}$ such that $(w_{i},R_{n_i},S_{n_i})$ is a weakly breaking triple for $H$ and $w_{i+1}(R_{n_{i+1}})\subseteq R_{n_i}$ for all $i\in \Nb$. The sequence $((w_1\dots w_i,R_{n_i},S_{n_i}))_{i>0}$ is then an infinite breaking sequence for $S_{n_1}$ with $w_1\dots w_i\in H$ for all $i$. This is absurd in view of Lemma~\ref{lem:breaking_sequence}.
\end{proof}

\section{Commensurated subgroups of groups almost acting on trees} \label{sec-proofs-commens}

Using our results from the previous sections, we now obtain our results for commensurated subgroups of groups of almost automorphisms.

\subsection{Commensurated subgroups of $\aaut$}

\begin{thm}\label{thm-trichotomy}
	If $\Lambda\leq \aaut$ is commensurated, then either $\Lambda$ is finite, $\ol{\Lambda}$ is compact and open, or $\Lambda=\aaut$.
\end{thm}
\begin{proof}
Suppose that $\Lambda$ is not finite. If $\Lambda$ contains a translation, then one must have $\Lambda=\aaut$ thanks to Theorem~\ref{thm:no_trans}. We thus suppose that $\Lambda$ is elliptic, which is equivalent to asserting that $\Lambda$ is locally elliptic by Corollary \ref{cor:elliptic_sgrp}. Proposition~\ref{prop:contain U} gives a locally elliptic, commensurated subgroup $H$ such that $\aut\leq H$ and $|\Lambda:\Lambda\cap H|$ is finite, and applying Theorem~\ref{thm:non-comm} to $H$, we deduce that $H$ must be compact. The subgroup $\Lambda$ is thus relatively compact. 

The closure $\ol{\Lambda}$ is commensurated via Lemma~\ref{lem:commensurated_closure}, so we infer that $\ol{\Lambda}$ is an infinite compact commensurated subgroup. As the group $\aaut$ is abstractly simple, we are in position to apply Theorem \ref{thm-CRW2-finite-open} to conclude that $\ol{\Lambda}$ is a compact open subgroup.
\end{proof}

\subsection{Thompson's groups} \label{subsec-thompson}

We here restrict ourselves to the case $d=k=2$ for ease of discourse, but the results can be easily adapted to the (commutator subgroups of the) groups $F_{d,k}$, $T_{d,k}$, and $V_{d,k}$. 

\begin{prop} \label{prop-commens-F}
Every commensurated subgroup of Thompson's group $F$ is a normal subgroup of $F$. 
\end{prop}

\begin{proof}
Suppose that $\Lambda$ is a non-trivial commensurated subgroup of $F$. Since every non-trivial element of $F$ is a translation, Theorem \ref{thm:no_trans} implies that $\Lambda$ contains the derived subgroup of $F$ and is therefore normal in $F$. 
\end{proof}

\begin{prop} \label{prop-commens-T}
Every commensurated subgroup $\Lambda$ of $T$ is either finite or equal to $T$. 
\end{prop}

\begin{proof}
If the group $\Lambda$ contains a translation, then Theorem \ref{thm:no_trans} implies that $\Lambda$ must be equal to $T$. We thus suppose that $\Lambda$ is an elliptic subgroup. 

Since $\Lambda$ contains no translations, Corollary \ref{cor-torion-Vdk} implies every finitely generated subgroup of $\Lambda$ is finite, and hence cyclic since $\Lambda \leq T$. In particular, $\Lambda$ is abelian and locally cyclic. Taking the primary decomposition $\Lambda = \bigoplus \Lambda_p$, each $\Lambda_p$ is either a Pr\"{u}fer $p$-group or a finite cyclic $p$-group.

Suppose first that there is a prime $p$ such that $\Lambda_p$ is a Pr\"{u}fer $p$-group. The group $\Lambda_p$ thus has no proper finite index subgroup. Since $\Lambda$ is commensurated in $T$, we infer that $\Lambda_p \leq \Lambda \cap g \Lambda g^{-1}$ for every $g \in T$. In particular, $\Lambda$ contains a non-trivial normal subgroup of $T$, which is absurd because $T$ is simple.

It is thus the case that $\Lambda_p$ is finite cyclic for every prime $p$, i.e.\ $\Lambda$ is a direct sum of finite cyclic $p$-groups. Combining this decomposition with the fact that $\Lambda$ is commensurated, every $g \in T$ must indeed normalize a, necessarily finite index, subgroup of $\Lambda$ of the form $\bigoplus_{p\in \pi}\Lambda_p$ where $\pi$ includes all but finitely many primes. Since $T$ is finitely generated, we deduce the existence of a finite index subgroup $\Lambda'$ of $\Lambda$ that is normalized by $T$. The group $T$ is simple,  so $\Lambda' = 1$. Therefore, $\Lambda$ is finite.
\end{proof}

\begin{prop} \label{prop-commens-V}
Every proper commensurated subgroup of $V$ is locally finite. 
\end{prop}

\begin{proof}
A proper commensurated subgroup of $V$ must contain only elliptic elements by Theorem \ref{thm:no_trans}, and such a subgroup must be locally finite via Corollary \ref{cor-torion-Vdk}. 
\end{proof}

Propositions \ref{prop-commens-F} and \ref{prop-commens-T} give a complete description of the commensurated subgroups of Thompson's groups $F$ and $T$. Although Proposition \ref{prop-commens-V} establishes some restrictions for the commensurated subgroups of $V$, we are not able to obtain a complete classification. 

There does exist an infinite commensurated subgroup in $V$, namely the group of finitary automorphisms of the tree $\mathcal{T}_{2,2}$; see \cite[Example 6.7, Proposition 7.11]{LB14}. The group $V$ thus has at least three commensurability classes of commensurated subgroups: the trivial group, the group of finitary automorphisms of $\mathcal{T}_{2,2}$, and the entire group $V$. We thus arrive at an interesting, open question.

\begin{quest} \label{quest-commens-V}
Does Thompson's group $V$ have more than three commensurability classes of commensurated subgroups?
\end{quest}

Thanks to the process of Schlichting completion \cite{schlich-comple}, a related problem is the following:

\begin{quest} \label{quest-dense-embed-V}
Are there non-discrete locally compact groups other than $\mathrm{AAut}(\mathcal{T}_{2,2})$ into which Thompson's group $V$ embeds densely ?
\end{quest}


The combination of Theorem \ref{thm:no_trans} and Corollary \ref{cor:elliptic_sgrp} may be applied to other interesting finitely generated subgroups of $\aaut$. One example of such a group is the finitely presented simple group $V_{\mathcal{G}}$ containing the Grigorchuk group $\mathcal{G}$ constructed by C. R\"{o}ver \cite{Rov-cfpsg}. Theorem \ref{thm:no_trans} shows any proper commensurated subgroup $\Lambda \lneq V_{\mathcal{G}}$ contains only elliptic elements, and Corollary \ref{cor:elliptic_sgrp} implies that any finitely generated subgroup of $\Lambda$ sits inside some permutational wreath product $\mathcal{G} \wr \mathrm{Sym}(n)$. In particular, since $\mathcal{G}$ is a torsion group, it follows that $\Lambda$ is torsion. 

Question \ref{quest-commens-V} can be posed for the group $V_{\mathcal{G}}$. Since the Grigorchuk group is commensurated in $V_{\mathcal{G}}$ \cite{Rov-cfpsg} (see also \cite[Proposition 7.11]{LB14}), there are at least three commensurability classes of commensurated subgroups in $V_{\mathcal{G}}$. We do not know whether there are more than three. 

\section{Applications}

Our study of commensurated subgroups concludes by considering several applications.  We first study embeddings into \tdlc groups.

\begin{prop} \label{prop-closed-image}
Let $G$ be a \tdlc group such that every proper commensurated open subgroup of $G$ is compact. Then every continuous homomorphism $\varphi :G\rightarrow H$ with $H$ a \tdlc group has closed image.
\end{prop}

\begin{proof}
Choose a compact open subgroup $U$ of $H$ and consider $\varphi^{-1}(U)$ in $G$. Since the preimage of a commensurated subgroup remains commensurated, $\varphi^{-1}(U)$ is commensurated in $G$. By continuity, $\varphi^{-1}(U)$ is open, so it must be either compact or equal to $G$. If $\varphi(G) \leq U$ for every $U$, then one has $\varphi(G)=1$ since compact open subgroups form a basis at $1$. Otherwise there is $U$ such that $\varphi^{-1}(U)$ is compact, and  
\[
\varphi(\varphi^{-1}(U)) = U \cap \mathrm{Im}(\varphi)
\]
is compact. The image $\mathrm{Im}(\varphi)$ is thus closed, and the desired result follows.
\end{proof}

Via Theorem~\ref{thm:aaut-intro}, the group $\aaut$ satisfies the assumption of Proposition \ref{prop-closed-image}.

\begin{cor} \label{cor-aaut-closed-im}
Every continuous homomorphism $\varphi : \aaut \rightarrow H$ with $H$ a \tdlc group has closed image.
\end{cor}

We emphasize that there are generalizations of the group $\aaut$ for which Corollary~\ref{cor-aaut-closed-im} does not hold. For a permutation group $D \leq \Sd$, we set $W(D)$ to be the natural profinite completion of the iterated permutational wreath products of copies of $D$. The group $\Daaut$ is then defined to be the subgroup of $\aaut$ consisting of almost automorphisms acting locally by an element of $W(D)$. When $D$ is trivial, $\Daaut$ is the Higman-Thompson group $V_{d,k}$. When $D = \mathrm{Sym}(d)$, we have $\Daaut = \aaut$. The group $\Daaut$ further admits a locally compact group topology. We refer to \cite{CdM11} for more details; see also \cite{LB14} and \cite{Sauer-Thu}. Corollary~\ref{cor-aaut-closed-im} \textit{does not} hold in general for the groups $\Daaut$. Indeed, the embeddings $\Daaut \hookrightarrow \aaut$ are continuous and have dense image since $\Daaut$ always contains $V_{d,k}$, which is dense in $\aaut$. Corollary \ref{cor-aaut-closed-im} thus fails when $D$ is not equal $\mathrm{Sym}(d)$. 

\begin{proof}[Proofs of Corollaries~\ref{cor-intro-embed-T} and \ref{cor-intro-embed-F}] 
According to Proposition \ref{prop-commens-T}, the assumption of Proposition \ref{prop-closed-image} also holds true for the group $T$. Since any countable closed subgroup of a locally compact group must be discrete, this shows that any embedding of the group $T$ into a \tdlc group must have discrete image. This proves Corollary \ref{cor-intro-embed-T}. 

The proof of Corollary \ref{cor-intro-embed-F} follows the same lines. By applying Proposition \ref{prop-commens-F}, we obtain that any embedding of $F$ into a \tdlc group $H$ must have discrete image. Since $F$ is torsion-free, this implies that $F$ must intersect trivially any compact open subgroup of $H$.
\end{proof}

For our second application, we consider lattice embeddings. Lattice embeddings of discrete groups are of particular interest; cf.\ \cite{BFS15}. Our results place additional restrictions on such embeddings for Thompson's group $T$. 

Recall that by work of V. P. Platonov \cite{P_66}, every locally compact group admits a unique maximal locally elliptic normal subgroup, called the locally elliptic radical. This subgroup is also closed.

\begin{thm}\label{thm:latticesT}
Suppose that $G$ is a compactly generated locally compact group admitting $T$ as a lattice and denote by $R$ the locally elliptic radical of $G$. Then $R$ is compact, $G/R$ is a \tdlc group with a unique minimal non-trivial closed normal subgroup $H$, and $H$ satisfies the following properties:
\begin{enumerate}
	\item $H$ is a compactly generated topologically simple \tdlc group;
	\item $H$ is cocompact in $G/R$ and contains $T$ as a lattice.
\end{enumerate}  
\end{thm}

\begin{proof}
We first show that the connected component $G^{\circ}$ is compact. Let $O\leq G$ be an open subgroup of $G$ such that $O/G^{\circ}$ is compact. The group $O$ is a commensurated open subgroup of $G$, since it is the preimage of a compact open subgroup under the projection $G\rightarrow G/G^{\circ}$. Suppose for contradiction that $O$ is not compact. That $T$ is a lattice implies $O\cap T$ is infinite, and furthermore, $O\cap T$ is a commensurated subgroup of $T$. Applying Proposition~\ref{prop-commens-T}, we deduce that $T\leq O$. The quotient $O/G^{\circ}$ is a profinite group, and since $T$ admits no finite quotients, we indeed have $T\leq G^{\circ}$. This is absurd, since finitely generated infinite simple groups are never lattices in connected locally compact groups. The connected component of $G$ is thus compact. By passing to $G/G^{\circ}$, we assume that $G$ is a \tdlc group.

Let $U\leq G$ be a compact open subgroup and let $N\normal G$ be a non-compact closed normal subgroup. Applying Lemma~\ref{lem:commensurated_and_normal}, the subgroup $UN$ is a commensurated subgroup of $G$. Proposition~\ref{prop-commens-T} thus implies that in fact $T\leq UN$. Letting $(U_i)_{i\in I}$ be a basis at $1$ of compact open subgroups, it follows that $T\leq \bigcap_{i\in I}U_iN=N$. The quotient $G/N$ is therefore a locally compact group which admits an invariant probability measure, and thus, it is a compact group. We deduce that every non-compact normal subgroup of $G$ is cocompact in $G$ and contains $T$. In particular, the locally elliptic radical $R\normal G$ is compact. Indeed, if $R$ is cocompact in $G$, then $G$ is locally elliptic via \cite{P_66}, so $T$ is a torsion group, which is absurd.

Setting $G' := G/R$, the group $G'$ has no non-trivial compact normal subgroups, and applying the previous paragraph, all non-trivial normal subgroups of $G'$ are cocompact in $G'$ and contain $T$. Consider $H$ the intersection of all non-trivial closed normal subgroups of $G'$. The group $H$ contains $T$ and is thus a non-trivial cocompact normal subgroup of $G'$. It remains to show that $H$ is topologically simple. Take $N$ a non-trivial closed normal subgroup of $H$. The subgroup $N$ cannot be compact, since otherwise $G'$ has a non-trivial locally elliptic radical. The above argument thus ensures that $N$ contains $T$ and is cocompact in $H$. The intersection $K$ of all non-trivial closed normal subgroups of $H$ is then a non-trivial characteristic subgroup of $H$. By our choice of $H$, we deduce that $K=H$, and a fortiori, $N=H$. That is to say, $H$ is topologically simple.
\end{proof}


\bibliographystyle{abbrv}
\bibliography{biblio1}

\end{document}